\documentclass[reqno]{amsart}

\usepackage{hyperref}
\usepackage[dvipsnames]{xcolor}

\usepackage{amsmath,amsfonts,amssymb,amsthm,amscd}

\usepackage{enumerate}
\usepackage{color,subfigure}
\usepackage{multicol}
\usepackage{float}

\usepackage{verbatim}

\usepackage[noabbrev]{cleveref}

\definecolor{chianti}{rgb}{0.6,0,0}
\definecolor{meretale}{rgb}{0,0,.6}
\definecolor{leaf}{rgb}{0,.35,0}

\DeclareMathOperator{\Ann}{Ann}

\DeclareMathOperator{\Hom}{Hom}

\DeclareMathOperator{\rank}{dim}
\DeclareMathOperator{\edim}{edim}

\DeclareMathOperator{\Tor}{Tor}

\DeclareMathOperator{\depth}{depth}
\DeclareMathOperator{\reg}{reg}

\DeclareMathOperator{\Ext}{Ext}

\newcommand{\A}{\mathbb{A}}

\newcommand{\sube}{\subseteq}

\renewcommand{\phi}{\varphi}

\renewcommand{\P}{\mathbb{P}}

\renewcommand{\H}{{\rm H}}
\newcommand{\G}{{\rm G}}
\newcommand{\T}{{\rm T}}
\newcommand{\B}{{\rm B}}
\newcommand{\C}{{\rm C}}
\newcommand{\GT}{{\rm GT}}
\newcommand{\GH}{{\rm GH}}
\newcommand{\GS}{{\rm GS}}
\renewcommand{\S}{{\rm S}}

\usepackage{fbox}

\theoremstyle{plain} 
\newtheorem{theorem}{Theorem}[section] 
\newtheorem{corollary}[theorem]{Corollary}
\newtheorem{lemma}[theorem]{Lemma}

\newtheorem*{question*}{Question}

\newtheorem{prop}[theorem]{Proposition}

\newtheorem{quest}[theorem]{Question}

\theoremstyle{definition}
\newtheorem{definition}[theorem]{Definition}

\newtheorem{example}[theorem]{Example}
\newtheorem{notation}[theorem]{Notation}

\theoremstyle{remark}
\newtheorem{remark}[theorem]{Remark}

\usepackage[dvipsnames]{xcolor}

\title[Tor algebras classified by Waring rank]{Classification of Tor algebras for rings from general points by Waring Rank}

\author{Kara Fagerstrom}
\address{Department of Mathematics, University of Nebraska--Lincoln, 203 Avery Hall}
\email{kfagerstrom2@huskers.unl.edu}

\begin{document}

\begin{abstract}
    We determine the multiplicative structure on the Tor algebras associated to the coordinate rings of general
points in $\mathbb{P}^3$ in terms of the number of points. 
We show that  in codimension 
four, the Waring rank of a form   determines the Tor algebra class of its apolar Artinian Gorenstein ring for a dense subset of the forms of fixed Waring rank and sufficiently large degree.
\end{abstract}

\maketitle

\section{Introduction}

The structure of minimal free resolutions plays a central role in commutative algebra and algebraic geometry.  Tor algebras provide a multiplicative framework that enriches this homological landscape. For a quotient $R$ of a regular ring $Q$ with residue field $k$  the $k$-algebra $\Tor_Q(R,k)$ is termed the Tor algebra of $R$.

In low
codimension, the multiplicative structure on the Tor algebra of a local or graded
ring admits a remarkably rigid classification. Foundational work of
Buchsbaum and Eisenbud \cite{BuchsbaumEisenbud} established that minimal free resolutions of rings of
codepth at most three admit differential graded algebra structures, and subsequent
work of Weyman \cite{Weyman} and Avramov--Kustin--Miller \cite{AKM-codepth3} classified the possible
multiplicative structures on the Tor algebras arising in these settings. In codimension four,
Kustin and Miller \cite{Kustin_Miller_Codepth4class} extended this classification to Gorenstein rings. More recent results around these themes include \cite{Avramov}, \cite{ChristensenVeliche14}, \cite{ChristensenVelicheWeyman19},  \cite{ChristensenVelicheWeyman20}, \cite{VandeBogert21}, \cite{VandeBogert22}, \cite{ChristensenVeliche23} and \cite{Hardesty}. Recent structure theorems in codimension four as in \cite{Gor_Six_Gen} provide novel insights. 

Despite these classification results, relatively little is known about how the
multiplicative structure of the Tor algebra behaves in naturally occurring
families of rings. Two particularly rich sources of examples arise from coordinate rings of finite sets of
general points in the projective space $\P^3$ (in codimension three) and from Artinian Gorenstein rings defined via
Macaulay inverse systems (in codimension four). 
Macaulay inverse systems provide a correspondence between
Artinian Gorenstein rings and homogeneous forms. Every homogeneous polynomial $F$ of degree $s$ gives rise to a unique associated apolar algebra which is Artinian Gorenstein, and conversely every graded Artinian Gorenstein ring arises in this manner. 
The decomposition of $F$ as a sum of  powers of linear forms 
\begin{equation}\label{eq: intro F}
F= \ell_1^{s} + \cdots + \ell_t^{s},
\end{equation}
determines its
Waring rank. Specifically the Waring rank of $F$ is the smallest integer $t$ such that an identity of the form \eqref{eq: intro F} exists for some linear forms $\ell_i$.
The two families of rings considered above are closely related by associating to a finite set of points
$X \subset \mathbb{P}^3$ a homogeneous form $F_X$ via \eqref{eq: intro F}, where the linear forms $\ell_i$ correspond to the points of $X$. This construction links
the geometry of points in projective space with families of Artinian Gorenstein
rings of prescribed Waring rank.

The purpose of this paper is to study the multiplicative structure on the Tor
algebras associated to these families and to show that, in codimension four, the Waring rank determines the Tor algebra class for a dense subset of the forms of fixed Waring rank.
More precisely, we classify the Tor algebras of the coordinate rings of general
points in $\mathbb{P}^3$, and we extend this classification to general
Artinian Gorenstein rings of codimension four obtained from Macaulay dual
generators of fixed Waring rank. In what follows we use the notation $\mu(t) =\min\left\{j\ge1:\binom{j+3}{3}>t\right\}$.

\begin{theorem}\label{full_theorem}

The Tor algebras associated to the coordinate ring $R$ of $t$ general points in $\P^3$ can be classified by the value of $t$. Likewise, the Tor algebra associated to a general Artinian Gorenstein ring $S$ of codimension 4 and Waring rank $t$ satisfying $2\mu(t)\leq\reg(S)\leq 3\mu(t)-4$ or $\reg(S)\geq 3\mu(t)+1$ can be classified by the value of $t\geq 4$ as given below, where the isomorphism classes in the table  are described in  \cref{s: Tor classification}.

\begin{center}

    \begin{tabular}{|c|c c c c c c|}
    \hline
    $t=$ & 4&5&6&7&8&$\geq9$ \\  
    \hline
    $R$ &   H(0,0)&  G(5)&  H(3,2)&  T&  H(1,0)&  H(0,0) \\
    $S$ & GS&GH(5)&G(5)&GT&GH(1)&GS\\
    \hline
  \end{tabular}
\end{center}
\end{theorem}

Our approach combines several ingredients. First, we  derive sharp bounds on
the possible ranks of multiplication in a Tor algebra and introduce the
notion of \emph{maximum multiplication rank}, which describes when these bounds
are attained. We then analyze the graded Betti
tables arising from the Minimal Resolution Conjecture and deduce that the Tor algebras for the coordinate rings of general sets of points in $\P^3$ satisfy maximum multiplication rank. For the codimension four Gorenstein case, we employ doubling
constructions and Macaulay inverse systems to relate the Tor algebra of an
Artinian Gorenstein ring to that of the coordinate ring of a set of points. This connection allows us to transfer some information about multiplicative
structures from codimension three to codimension four and once again to show that the corresponding Tor algebras enjoy maximum multiplication rank. The classification shown in \Cref{full_theorem} is then derived as a consequence of this property.
In summary, the classification reveals a strong parallel between the geometry
of general points in projective space and the homological structure of the
corresponding apolar algebras.

\paragraph{\bf Acknowledgements} The author received funding from NSF RTG grant DMS--2342256 Commutative Algebra at Nebraska and NSF grant DMS--2401482. Computations in Macaulay2 \cite{M2} performed with the package \cite{TorAlgebraSource} as described in the accompanying paper \cite{TorAlgebraArticle} were essential for the development of this paper.

\section{Background}

Throughout the paper we fix an algebraically closed field $k$ of characteristic zero. The characteristic assumption is needed to apply results on Betti numbers of general points in $\P^3$ from \cite{BGgenPts}. The field being algebraically closed allows to express arbitrary $k$-linear combinations of powers of linear forms as sums of such powers as well as ensuring the conditions for some of the classification theorems on Tor algebras in \cite{Kustin_Miller_Codepth4class} and \cite{Gor_Six_Gen} are met.

\subsection{Multiplicative structures on $\Tor$ algebras of low codimension}\label{s: Tor classification}

Let $R$ be a commutative Noetherian  
local or graded
ring with residue field $k$ which is a quotient of a regular local or graded ring $Q$. Assume $R=Q/I$ with $I\sube \mathfrak{m}^2$ where $\mathfrak{m}$ is the maximal ideal of $Q$.
Let 
\[
F= 0\to F_c\to \cdots \to F_1\to  F_0\to 0
\]
be a minimal free resolution for $R$ over $Q$. By the Auslander-Buchsbaum theorem $c=\edim R-\depth R$ is the codepth of $R$.  When $R$ is Cohen-Macaulay, $c$ coincides with the codimension of $R$. Buchsbaum and Eisenbud \cite{BuchsbaumEisenbud} proved that $F$ has a differential graded (DG) algebra structure when $c\leq 3$, which is not necessarily unique. 

The Tor algebra $H(F\otimes_Q k)=\Tor^Q(R,k)$ is also a DG algebra. For some classes of rings, the multiplicative structure on the Tor algebra can be classified up to isomorphism.
When $c\leq 3$ the multiplicative structures of the Tor algebra were determined by Weyman \cite{Weyman} and by Avramov, Kustin, and Miller \cite{AKM-codepth3}. Additionally the structures have also been classified  when $c=4$ and $R$ is Gorenstein, by Kustin and Miller \cite{Kustin_Miller_Codepth4class}.   
We describe the multiplication classes of the Tor algebras in these settings. 

The homology of $F\otimes_Q k$ is an exterior algebra if and only if $R$ is a complete intersection \cite[Theorem 2.7]{Assmus}. In this case $R$ is said to be of class $C(c)$. When $c=2$ and $R$ is not a complete intersection then $\Tor^Q(R,k)$ has trivial multiplication and $R$ belongs  to class $\rm S$.

We will use conventional notation as in \cite{Avramov} to describe the classes when $c=3$. 
\begin{notation}\label{notation_for_Tor_invariants}

    Set $c=\edim R-\depth R$.
    For 
   $ A=\Tor^Q(R,k)$ let
    \begin{alignat*}{2}
    {l} &=\rank_k A_1 &\quad p&=\rank_k(A_1\cdot A_1)\\
    m&=\rank_kA_2 &\quad q&=\rank_k(A_1\cdot A_2)\\
    n&=\rank_k A_3 &\quad   r&=\rank_k(\delta:A_2\to \Hom(A_1,A_3)).
\end{alignat*}
    
\end{notation}

Let $A$ have basis $\{e_1,\ldots,e_{l}\}$ 
for $A_1$, $\{f_1,\ldots,f_m\}$ for $A_2$ and $\{g_1,\ldots,g_n\}$ for $A_3$.
The ring $R$ can be classified into one of the following isomorphism classes with the nonzero multiplication for $A_1\cdot A_1$ and $A_1\cdot A_2$ given: 

\begin{figure}[h!]    
\renewcommand{\arraystretch}{1.2}
\begin{tabular}{|l|c c c |c c|}
    \hline
Class & $p$& $q$ & $r$ & $A_1\cdot A_1$ & $A_1\cdot A_2$ \\ 
\hline
 $\C(3)$ & 3 & 1 & 3 &
         $e_1e_2=f_3$, $e_3e_1=f_2$, $e_2e_3=f_1$ & 
       $e_if_i=g_1$ for $1\leq i\leq 3$ \\ 
$\T$ & 3 & 0 & 0 &
    $e_1e_2=f_3$, $e_3e_1=f_2$, $e_2e_3=f_1$ & 
    0 \\ 
$\B$ & 1 & 1 & 2 &
    $e_1e_2=f_3$ & 
    $e_1f_1=g_1$, $e_2f_2=g_1$ \\
$\G(r)
$ & 0 & 1 & $r$ &
    0 &
    $e_if_i=g_1$ for $1\leq i\leq r$ \\ 
$\H(p,q)$ & $p$ & $q$ & $q$ &
    $e_{p+1}e_i=f_i$ for $1\leq i\leq p$ &
    $e_{p+1}f_{p+j}=g_j$ for $1\leq j\leq q$\\
    \hline
\end{tabular}
\caption{Tor algebras of codimension 3 rings.}
\label{fig: codim 3}
\end{figure}

A ring $R$ is Gorenstein if and only if it is Cohen-Macaulay and $n=1$. Gorenstein rings are either of class $\C(3)$ if $l=3$ 
or of class $\G(l)$ if $l>3$. 
The values of $p,q$ and $r$ are distinct for all classes except $\T$ and $\H(3,0)$. 
So in general the class of $R$ can be determined by the values of $p,q$ and $r$ shown in the table above. We leverage this approach in \Cref{ClassThmForJ}.

To distinguish between the classes $\T$ and $\H(3,0)$ it is convenient to use the following result.

\begin{lemma}\label{aci-class} \cite[Proof of Thm 2.]{Avramov-ACI}
    If $R$ is a Cohen-Macaulay almost complete intersection of codepth 3 it can be classified by type $n=\rank_k A_3$
    \begin{enumerate}
        \item Class $\H(3,2)$ when $n=2$.
        \item Class $\T$ when  $n\geq 3$ and odd.
        \item Class $\H(3,0)$ when $n\geq 4$ and even.
    \end{enumerate}
\end{lemma}

For Gorenstein rings $R$ of codepth 4
we will use the following notation:
\begin{notation}\label{Notation Gor tor invariants}
    For $c=4$, 
   let $ A=\Tor^Q(R,k)$, $m=\rank_k A_1-1$,
    \begin{alignat*}{2}
     \bar p&=\rank_k(A_1\cdot A_1),  &\quad \bar q&=\rank_k(A_1\cdot A_2).
\end{alignat*}
    
\end{notation}

Let $A$ have basis $\{e_1,\ldots,e_{m+1}\}$ for $A_1$, $\{f_1,\ldots,f_m,f_1',\ldots,f_m'\}$ for $A_2$, $\{g_1,\ldots,g_{m+1}\}$ for $A_3$, and $\{h\}$ for $A_4$. As $A$ is Gorenstein it exhibits 
Poincare duality: $e_ig_j=\delta_{ij}h$, $f_if_j'=\delta_{ij}h$, $f_if_j=f_i'f_j'=0$ for all $i$ and $j$.  

The ring $R$ can be classified into one of the following isomorphism classes with the  nonzero multiplication for $A_1\cdot A_1$ and $A_1\cdot A_2$ given. The nomenclature reflects the terminology employed in the Macaulay2 package \cite{TorAlgebraSource}.

\begin{figure}[h!]
\renewcommand{\arraystretch}{1.2}
\begin{tabular}{|r|c c| }
    \hline
Class & $A_1\cdot A_1$ & $A_1\cdot A_2$ \\
\hline
 $\C(4)$& 
        Exterior algebra & \\     
 $\GT$ & $e_1e_2=f_1$,$e_1e_3=f_2$, $e_2e_3=f_3$ &
    $e_1f_2'=e_2f_3'=-g_3$, \\
    & & $e_2f_1'=e_3f_2'=g_1$, $-e_1f_1'=e_3f_3'=g_2$ \\    
$\GH(\bar p)$ &  $e_{\bar p+1}e_i=f_i$ for $1\leq i\leq \bar p$ &
    $e_{\bar p+1}f_{i}'=-g_i$, $e_if_i'=g_{\bar p+1}$ for $1\leq i\leq \bar p$\\
    $\GS$ & All zero & \\
        \hline
\end{tabular}
\caption{Tor algebras of Gorenstein codimension 4 rings.}
\label{fig: codim 4}
 \end{figure}

 \begin{remark}\label{Mult_duality}
      The classification is determined by the multiplication in $A_1\cdot A_1$ as they are distinct. From Poincare duality one can determine the multiplication in $A_1\cdot A_2$ from the multiplication in $A_1\cdot A_1$ and vice-versa  as  $$e_ie_j=f_k\iff e_jf_k'=g_i \iff -e_if_k'=g_j.$$ 
 \end{remark}

 The class of $R$ can be determined by the values of $\bar p$ and $\bar q$ as given below:
\begin{center}

\begin{tabular}{r|c c c c}
    Class &  C(4) &GT& $\overset{}{\GH(\bar p\geq 1)} $ &GS  \\ \hline
    $\bar p$ & 6&  3& $\bar p$& 0\\ 
    $\bar q$ & 4&3 &$\bar p+1$&0.\\
     
\end{tabular}
\end{center}
We leverage this in \Cref{AG codepth 4}.

\subsection{General points in $\mathbb{P}^3$ and their resolutions}

Let  $N>0$ be an integer and $Q=k[x_0,\ldots, x_N]$ be the coordinate ring of $\mathbb{P}^N$. A property holds for $t$ general points in $\mathbb{P}^N$ if it holds for every element of some nonempty Zariski open set of $(\mathbb{P}^N)^t$. Set 
\begin{eqnarray}
\mu(t) &=&\min\left\{i\ge1:\dim_k Q_i>t\right\} \label{eq: mu}\\
  d(t)&=&\min\{i\mid \dim_k Q_i\geq t\}. \label{eq: dt} 
\end{eqnarray}

The Minimal Resolution Conjecture (MRC) proposed by Lorenzini \cite{Lorenzini1, Lorenzini} 
concerns Betti numbers of the coordinate ring $R$ of $t$ general points defined as $$b_{i,j}=\dim_k \Tor_i^Q(R,k)_j.$$ It states that for $t$ general points in $\P^N$ 
the Betti numbers satisfy $b_{ij}=0$ when $j-i\notin \{ \mu(t)-1,\mu(t)\}$  and $b_{i,i+\mu(t)}\cdot b_{i+1,i+\mu(t)}=0$ for all $i$. Consequently the Betti table of $R$ has  at most two non-zero rows, labelled $\mu(t)$ and $\mu(t)-1$.  For $N=3$ and $t\geq 4$ the Betti table has the following shape:
\begin{equation}\label{Betti_codim3}
\begin{tabular}{r c c c c c}
 & 0 & 1 & 2 & 3 \\

total: & 1 & ${l}$ & $m$ & $n$ \\
0 : & 1 & . & . & . \\
$\mu(t)-1$ : & . & $\alpha_1$ & $\alpha_2$ & $\alpha_3$ \\
$\mu(t)$ : & . & $\beta_1$ & $\beta_2$ & $\beta_3$ 
\end{tabular}
\end{equation}

The regularity of a ring is $\max\{j-i\mid b_{ij}\neq 0\}$. 
For general sets of $t$ points,  $\mu(t)$ is the least degree of a relation of $R$, while $d(t)=\reg(R)$; see \cite[Corollary 4.7]{EisenbudSyzygy}. 

Although the MRC has been disproved for $N\geq 6$, $N\neq 9$, see \cite{MRC-counter}, Ballico and Geramita \cite{BGgenPts} proved the MRC is true for general points in $\P^3$. Additionally they gave formulas for the Betti numbers. The Betti tables for $4\leq t\leq 9$, i.e. when $\mu(t)=2$, are given below:

\begin{figure}[H]
\begin{center}
\begin{multicols}{3} 

$t=4$

\begin{tabular}{r c c c c}
 & 0 & 1 & 2 & 3 \\
total: & 1 & 6 & 8 & 3\\
0 : & 1 & . & . & . \\
1 : & . & 6 & 8 & 3 \\
2 : & . & . & . & . 
\end{tabular}

\columnbreak

$t=5$

\begin{tabular}{r c c c c}
 & 0 & 1 & 2 & 3 \\
total: & 1 & 5 & 5 & 1 \\
0 : & 1 & . & . & . \\
1 : & . & 5 & 5 & . \\
2 : & . & . & . & 1 
\end{tabular}

\columnbreak

$t=6$

\begin{tabular}{r c c c c}
 & 0 & 1 & 2 & 3 \\
total: & 1 & 4 & 5 & 2 \\
0 : & 1 & . & . & . \\
1 : & . & 4 & 2 & . \\
2 : & . & . & 3 & 2 
\end{tabular}
\end{multicols}

\begin{multicols}{3}

$t=7$

\begin{tabular}{r c c c c}
 & 0 & 1 & 2 & 3 \\
total: & 1 & 4 & 6 & 3 \\
0 : & 1 & . & . & . \\
1 : & . & 3 & . & . \\
2 : & . & 1 & 6 & 3 
\end{tabular}

\columnbreak

$t=8$

\begin{tabular}{r c c c c}
 & 0 & 1 & 2 & 3 \\
total: & 1 & 6 & 9 & 4 \\
0 : & 1 & . & . & . \\
1 : & . & 2 & . & . \\
2 : & . & 4 & 9 & 4 
\end{tabular}

\columnbreak

$t=9$

\begin{tabular}{r c c c c}
 & 0 & 1 & 2 & 3 \\
total: & 1 & 8 & 12 & 5 \\
0 : & 1 & . & . & . \\
1 : & . & 1 & . & . \\
2 : & . & 7 & 12 & 5 
\end{tabular}
\end{multicols}

\end{center}
\caption{Betti tables for coordinate rings of general points in $\P^3$.}\label{fig_Betti_tables}
\end{figure}

\subsection{Macaulay dual generators} 
\label{sect:MDG}

Let $Q=k[x_0,\ldots,x_N]$ be a polynomial ring and let $Q'=k[X_0,\ldots,X_N]$. The ring $Q'$ can be
regarded as a $Q$-module with action 

\[
x_i\circ F(X_0,\cdots,X_N)=\frac{\partial F}{\partial X_i}.
\]
 We regard $Q'$ as a graded $k$-algebra with $\deg X_i=\deg x_i$.
For each degree $i\geq 0$, the action of $Q$ on $Q'$ defines a non-degenerate $k$-bilinear pairing 
\begin{equation*}
\label{eq:MDPairing}
    Q_i \times Q'_i \longrightarrow k \text{ with } (f,F) \longmapsto f \circ F.
\end{equation*}
This gives for each $i\geq 0$ an isomorphism of $k$-vector spaces $Q'_i\cong \Hom_k(Q_i,k)$ via $F\mapsto\left\{f\mapsto f\circ F\right\}$ and thus $Q'=\bigoplus_{i\geq 0}\Hom_k(Q_i,k)$ is the graded dual of $Q$.

 It is a classical result of Macaulay \cite{Macaulay} (cf. \cite[Lemma 2.14]{IK}) that an Artinian $k$-algebra $S=Q/I$ is Gorenstein with socle degree $s$ if and only if  $I=\Ann_Q(F)=\{f\in Q\mid f\circ F=0\}$ for some homogeneous polynomial $F\in Q'_s$.  Moreover, this polynomial, termed a {\em Macaulay dual generator} for $S$, is unique up to a scalar multiple and $S$ is called {\em apolar} to $F$. Here by socle degree we mean $s=\max\{i\mid S_i\neq 0\}$ and this invariant is also known as the regularity of $S$.

\subsection{Doublings}

Let us recall the doubling construction which is needed later on.

\begin{definition}
Let $Q$ be a standard graded regular ring of dimension $N+1$. The \emph{canonical module} of a Cohen-Macaulay $Q$-module $M$ is $\omega_M = \Ext^{\dim Q - \dim M}_Q (M, \omega_Q)$ where $\omega_Q=Q(-N-1)$. 
\end{definition}

\begin{definition}\label{doubling}  
Let $R$ be a Cohen-Macaulay ring of dimension $d$ and $\omega_R$ its canonical module. Additionally, assume $R$ satisfies $G_0$ (i.e., it is Gorenstein at all minimal
primes). A ring $S$ of dimension $d-1$ is called a {\em doubling} of $R$ via $\psi$ if there exists  an integer $a>0$ and a short exact sequence of $R$-modules
\begin{equation}\label{eq:doubling}
0 \rightarrow \omega_{R}(-a)\stackrel{\psi}{\rightarrow} R \rightarrow S\rightarrow 0.
\end{equation}
By \cite[Proposition 3.3.18]{BH}, if $S$ is a doubling, then it is a Gorenstein ring.

\end{definition}

\section{General Artinian Gorenstein Rings of Fixed Waring Rank}

In this section we use apolarity to construct Artinian Gorenstein rings associated to sets of points in projective space.
Let $X=\{p_1,\ldots,p_t\}$ be a set of $t$ points in $\P^N$ and write $Q=k[x_0,\ldots,x_N]$ for the coordinate ring of $\P^N$.
For each point $p_i\in \P^N$ pick a representative $p_i=[a_0:\cdots : a_N]$ and define a linear form in $Q'=k[X_0,\ldots,X_N]$ by
$\ell_i=a_0X_0+\cdots +a_NX_N$. 
We denote
\begin{equation}\label{eq: FX}
F_X=\sum_{i=1}^t c_i \ell_i^{s} \qquad \text{ with } c_i\neq 0 \text{ for } 1\leq i\leq t.
\end{equation}
Since $k$ is algebraically closed the coefficients can be absorbed into the linear forms at the expense of changing representatives for the points $p_i$. This allows to take $F_X=\sum_{i=1}^t \ell_i^{s}$ moving forward.

 Using Macaulay inverse systems we can construct for each such form a Gorenstein ring $Q/\Ann(F_X)$  of socle degree $s$. 
 We shall see in \Cref{isDoubling}  below  that  these Gorenstein rings share the same homological properties regardless of the choice of representatives for $p_1, \ldots, p_t$ chosen above,  whence we abuse notation and write $F_X$ to highlight only the dependence on the set $X$.

The construction above 
corresponds to a regular map 
\begin{equation}\label{eq: psi}
\Psi_{t,s}: (\A^{N+1})^t\to Q'_s, \quad \Psi_{t,s}(X)=\ell_1^s+\cdots+\ell_t^s.
\end{equation}
The union over $t\geq 1$ of the images of  the  maps $\Psi_{t,s}$ covers  $Q'_s\setminus\{0\}$,
which leads to the following definition. 

\begin{definition}\label{def: Waring}
    The \emph{Waring rank} of a homogeneous element $F\neq 0$ of degree $s$ in a polynomial ring is the least positive integer $t$ such that $F=\ell_1^s+\cdots+\ell_t^s$ with $\ell_1,\ldots,\ell_t$ linear forms.  
\end{definition}

We turn our attention to sets of forms of a fixed degree and Waring rank.

\begin{notation}\label{W sets}
Let $W_{t,s}$ denote the set of forms in $Q'_s$ of Waring rank $t$ and  let  $W_{\leq t,s}=\rm{Im}(\Psi_{t,s})$ for the map $\Psi_{t,s}$ introduced in \eqref{eq: psi}. Then $W_{\leq t,s}$ represents the set of  forms in $Q'_s$ of Waring rank at most $t$. \end{notation}

In tensor theory the Zariski closure $\overline{ W_{\leq t,s}}$ is called the set of forms of {\em border rank at most $t$} and its projectivization  $\P\overline{ W_{\leq t,s}}$ is a classical object in algebraic geometry, the $t$-th secant variety of the $s$-uple Veronese embedding  of $\P^N$.

The goal of the section is to define general Artinian Gorenstein rings of specified Waring rank to correspond to general sets of points in $\P^N$.

\begin{definition}\label{general AG}
Let $\P U$ be a nonempty Zariski open subset of $(\P^N)^t$ and $U$ the preimage of $\P U$ in $(\A^{N+1})^t$. We define a {\em general graded Artinian Gorenstein ring of codimension $N+1$, Waring rank $t$ and socle degree $s$} to be any ring $Q/\Ann(F)$ whose Macaulay dual element satisfies $\Ann(F)_1=0$ (i.e., $F$ is concise) and
\begin{equation}\label{eq: general}
F\in W_{t,s}\cap \Psi_{t,s}(U).
\end{equation}

If $N=3$ and $\P U$ is the Zariski open subset of $(\P^3)^t$ consisting of points whose coordinate rings have Betti table as predicted by the Minimal Resolution Conjecture then we call the Artinian ring specified by $F$ as in \eqref{eq: general} {\em general with respect to Betti numbers}.
\end{definition}

In the next proposition we justify the terminology general by showing that the Macaulay dual forms of algebras in \Cref{general AG} form a (not necessarily open) dense subset of $W_{t,s}$, which contains a dense open set $D$ of $W_{t,s}$. 
This  is a departure from the usual convention that general sets are dense open subsets of algebraic varieties.  Since our results apply to all forms in \eqref{eq: general}, including those in $D$ (a general set in the conventional sense), we opt to take the larger set \eqref{eq: general} as our definition.

Before we discuss subsets of $W_{t,s}$,  it is worth clarifying its topological properties.
Recall that a set $C$ is {\em constructible} if it is a finite union of locally closed sets in the Zariski topology, that is, $C=\bigcup_{i=1}^m C_i\cap O_i$ where each $C_i$ is closed and $O_i$ is open.
 We recall below the well-known fact that $W_{t,s}$ is a constructible subset of $\mathbb A(Q'_s)$. 
 
\begin{prop}\label{general is dense}
Adopt the notation in \Cref{eq: psi} and \Cref{W sets} and suppose $U$ is  a nonempty Zariski open subset of $(\A^{N+1})^t$. Let $C_s$ denote the set of concise forms of degree $s$. If $t\geq N+1$ and $t(N+1)\leq \binom{N+s}{s}$, the set  $C_s\cap W_{t,s}\cap \Psi_{t,s}(U)$  contains a dense open subset  of  the nonempty constructible set $W_{t,s}$.
\end{prop}

\begin{proof}

For all positive integers $i$, the sets $W_{\leq i,s}$ are constructible by Chevalley's theorem as they are the images of the morphisms of finite type  $\Psi_{i,s}$ in \eqref{eq: psi}. Since 
\[
W_{t,s}=W_{\leq t,s}\setminus W_{\leq t-1,s}=W_{\leq t,s}\cap W_{\leq t-1,s}^c
\]
 we conclude that $W_{t,s}$ is constructible.  
  
 The Zariski closure $\overline{W_{\leq t,s}}$ is irreducible of dimension 
 \[
 \dim\overline{W_{\leq t,s}}= \min\left \{t(N+1),\binom{N+s}{s}\right \}
 \]
  except for the cases $s=2$ and  $(t,s,N)=(5,4,2), (9,4,3), (14,4,4), (7,3,4)$ where the dimension is known explicitly; see \cite[Theorem 1.6]{IK}. It follows from the hypothesis that $\dim\overline{W_{\leq t-1,s}}< \dim\overline{W_{\leq t,s}}$, thus $W_{t,s}=W_{\leq t,s}\setminus W_{\leq t-1,s}$ is nonempty. 
  
  By \cite[Theorem 10.19]{GW}, as $W_{t,s}$ is constructible, it contains a subset $D'\subseteq W_{t,s}$ such that $D'$ is a dense, hence nonempty, open subset of $\overline{W_{t,s}}$. 

For $F\in Q'_s$, consider the first
catalecticant map
$\operatorname{Cat}_{1,F}:Q_1\rightarrow Q'_{s-1}$, $q\mapsto q\circ F$.
By definition, $\ker(\operatorname{Cat}_{1,F})=\Ann_Q(F)_1$.
Thus $F$ is concise if and only if 
$\operatorname{Cat}_{1,F}$ has rank $N+1$, hence $C_s$ is a Zariski open set. It is nonempty as $\sum_{i=0}^N X_i^s\in C_s$.
  
  Put $U'=\Psi_{t,s}^{-1}(D'\cap C_s)$ and observe this is a nonempty open set as both $D'$ and $C_s$ are so. Replacing $U$ by its subset $U\cap U'$, which is a nonempty thus dense open set of $(\A^{N+1})^t$, we may now assume that $\Psi_{t,s}(U)\subseteq C_s\cap W_{t,s}$. Consequently we have
  \[
 C_s\cap  W_{t,s} \cap \Psi_{t,s}(U)=\Psi_{t,s}(U).
  \]
 Since $\Psi_{t,s}(U)$ is constructible by Chevalley's theorem, it contains a dense open subset $D$. We claim that $D$ is also dense in $W_{t,s}$. Indeed, as the image of a dense set through a continuous map is a dense set in the image of the map and since the image of $\Psi_{t,s}$ is $W_{\leq t,s}$ we have that $\overline{\Psi_{t,s}(U)}=\overline{W_{\leq t,s}}$. In view of the containments $\Psi_{t,s}(U)\subseteq W_{t,s}\subseteq W_{\leq t,s}$ and the choice of $D$ this leads to the desired equalities
 \[
 \overline{D}=\overline{\Psi_{t,s}(U)}=\overline{W_{t,s}}=\overline{W_{\leq t,s}}.\qedhere
 \]
\end{proof}

\section{Main results}

In this section we establish bounds on the rank of the internal multiplication maps on some $\Tor$ algebras. Then we show that these bounds are attained for general inputs in low codimension, thereby establishing our main theorem. These bounds are based on possible multiplication by degree unlike in \cite{LinkageClasses}.

\begin{lemma}\label{General-bounds}
    Given a  bi-graded $k$-algebra $A$, the rank of the multiplication has upper bound
    \begin{equation}\label{alg_bound_eq}
        \dim_k ([A_i]_d\cdot [A_j]_e) \leq \min(\dim_k [A_i]_d\cdot \dim_k[A_j]_e, \dim_k[A_{i+j}]_{d+e}) .
    \end{equation}
When $A$ is a graded-commutative algebra of characteristic not equal to $2$ with an additional internal grading and $i$ is odd the bound can be improved 
\begin{equation}\label{DG_bound_eq}
    \dim_k ([A_i]_d\cdot [A_i]_d) \leq \min\left(\textstyle \binom{\dim_k [A_i]_d}{ 2}\displaystyle, \dim_k[A_{2i}]_{2d}\right) .
\end{equation}
When $A$ is graded-commutative algebra $k$-algebra of top degree $c$ that has Poincare duality, then for any $i,j,\ell$ such that $i+j+\ell =c$ we have  
\begin{equation}\label{exception Poincare}
\rank_k(A_i\cdot A_j)\neq0\iff \rank_k(A_i\cdot A_\ell)\neq0\iff\rank_k(A_j\cdot A_\ell)\neq0.
\end{equation}
\end{lemma} 

\begin{proof}
    The multiplication $A_i\otimes A_j\to A_{i+j}$  is a $k$-linear map so its rank is bounded above by the ranks of the source and target specifically in each internal degree. This gives the first inequality. In a DG-algebra if $i$ is odd $A_i\cdot A_i\sube \bigwedge^2A_i$ and since $$\dim_k \bigwedge ^2[A_i]_d= \binom{\dim_k[A_i]_d}{ 2}\leq \left(\dim_k[A_i]_d\right)^2,$$ combined with the first inequality this gives the second. 

    If $\rank_k(A_i\cdot A_j)\neq0$ there exists $x\in A_i$, $y\in A_j$ such that $xy\neq 0\in A_{i+j}$. As $i+j+\ell =c$, from Poincare duality there is some $z\in A_\ell$ such that $(xy)z \neq 0$ in $A_c$ therefore $xz\neq 0$ and $yz\neq 0$. This implies $\rank_k A_iA_\ell \neq 0$ and $\rank_k A_jA_\ell\neq0$.
\end{proof}

We single out the extremal algebras with respect to the preceding inequalities.

\begin{definition}
   A DG algebra has {\it maximum multiplication rank} if it attains equality in \eqref{DG_bound_eq} for every odd $i$ and every $d$ and it attains equality in \eqref{alg_bound_eq} for all cases not covered by \eqref{DG_bound_eq} subject to exceptions arising from \eqref{exception Poincare}.   
\end{definition} 
   
    \Cref{ex: codim 4 p q bounds} $t=9$ is an instance where the exceptions in \eqref{exception Poincare} are relevant.

\subsection{Tor Algebras in codimension three}

We continue by  studying Tor algebras of coordinate rings of points in $\P^3$ general with respect to their Betti numbers and we show that they have maximum multiplication rank. The Betti tables of such rings are outlined in  \eqref{Betti_codim3}, allowing us to put \Cref{General-bounds} in action.

\begin{lemma}\label{bounds_for_codim3}
    For a  ring $R$ with Betti table as in \eqref{Betti_codim3} the invariants for the associated Tor algebra defined in \Cref{notation_for_Tor_invariants} satisfy   $$\begin{cases}
    p\leq \min\left(\textstyle{\frac{\alpha_1(\alpha_1-1)}{2}},\beta_2\right), \, q\leq \min(\alpha_1\cdot\alpha_2,\beta_3) & \text{if }\reg(R)=2\\ 
    p=0,\, q=0 & \text{if }\reg(R)\neq 2. \end{cases}$$ 
\end{lemma}
 \begin{proof}
 First we will consider the $\reg(R)=2$ case. From \eqref{Betti_codim3} $A_1$ is concentrated in internal degrees 2 and 3 and $A_2$ is concentrated in internal degrees 3 and 4. Applying \Cref{General-bounds} with $i=j=1$ and $d=2$ and $e=3$ we get $\rank_k([A_1]_2\cdot [A_1]_3)\leq 0$ likewise for $i=1$ and $d=3$ we get $\rank_k([A_1]_3\cdot [A_1]_3)\leq 0$. So the nonzero multiplication occurs only when $i=1$ and $d=2$; applying \Cref{General-bounds} gives the first inequality. Similarly for $A_1\cdot A_2$ as $A_3$ is concentrated in degrees 4 and 5, the multiplication is nonzero only when $e+d=4,5$. This happens only for $[A_1]_2\cdot [A_2]_3\to  [A_3]_5$ which by \Cref{General-bounds} gives the second inequality.  If $\reg(R)=1$ or $\reg(R)>2$ the products of elements in $A_{\geq 1}$ must be zero by \Cref{General-bounds} as in both cases the multiplication will land in  internal degrees forced to be zero by the regularity.

 \end{proof}
 
 \begin{example}\label{ex: pq values}
Using the Betti tables in \Cref{fig_Betti_tables} one can calculate the maximum values of $p$ and $q$ allowed by \Cref{bounds_for_codim3}:

\begin{center}
\begin{tabular}{c|c c c c c c}
        $t=$& $4$&5&6&7&8&9\\
    \hline
     $ p \leq$& 0&0&3&3&1&0  \\
     $ q \leq$& 0&1&2&0&0&0
\end{tabular}
\end{center}
\end{example}

Interpreting the second case of \Cref{bounds_for_codim3}
 in conjunction with \Cref{fig: codim 3}  yields the following.

\begin{corollary}\label{reg 2}
A ring $R$ with Betti table as in \eqref{Betti_codim3} and $\reg(R)\neq2$ is in class $\rm H(0,0)$. 
\end{corollary}

The following result proves the codimension three case of our main theorem based on the discussion following \Cref{fig: codim 3}.

\begin{theorem}\label{ClassThmForJ}
   The coordinate ring $R$ of $t$ general points in $\P^3$  can be classified by the value of $t$ as given in \Cref{full_theorem}, and its associated Tor algebra has maximum multiplication rank. 
\end{theorem}

\begin{proof}
We use \Cref{notation_for_Tor_invariants} for $R=Q/J$, where $Q$ is the coordinate ring of $\P^3$ and $J$ is the defining ideal of $t$ general points in $\P^3$ with respect to Betti numbers.

When $t=4$ \Cref{ex: pq values} yields $p=q=0$ and thus from  \Cref{fig: codim 3} we infer that $R$ is in class $\H(0,0)$.

When $t=5$ by examining \Cref{fig_Betti_tables} we have $n=\rank_kA_3=1$, thus $R$ is Gorenstein, and ${l}=\rank_kA_1=5$, 
thus $R$ is not a complete intersection. By \cite[1.4.2.]{Avramov} $R$ is of class $\G({l})=\G(5)$.

When $t=6$ or 7 from \Cref{fig_Betti_tables} we have ${l}=4 $ 
so $R$ is an almost complete intersection. By
\Cref{aci-class}, when $t=6$, $R$ is in class $\H(3,2)$ and when $t=7$, $R$ is in class $\T$.

When $t=8$, from \Cref{fig_Betti_tables} the ideal $J$ has exactly two quadratic generators. Let $f, g$ be these generators and let $F$ denote a minimal free resolution of $Q/J$ over $Q$. Since $Q/J$ has codimension three, $F$ has a DG-algebra structure. Let $e_1,e_2$ be generators of $[F_1]_2$ such that $\partial(e_1)=f$ and $\partial(e_2)=g$. Consider the Koszul complex $K$ on $f,g$. There is a $DG$-algebra morphism $K\to F$ taking $Q$ to $F_0=Q$ identically and the generators of $K$ as an exterior algebra to $e_1, e_2$. This induces a DG algebra map in homology $\phi:H(K^k)\to A$, where $K^k=K\otimes_{Q}k$.

 The degree-four Koszul syzygy on $f,g$ is a minimal generator of the module of second syzygies on $Q/J$ since there are no degree-three second syzygies as evidenced by \Cref{fig_Betti_tables}; therefore its class in $A$ remains nonzero after tensoring with $k$. Since $\phi$ takes Koszul syzygies to Koszul syzygies $\phi(H_2(K^k))$ is not zero. Additionally $\phi(H_2(K^k))=\bar e_1\cdot \bar e_2$ where $\bar e_i$ is the image of $e_i$ in $A_1$. Thus we have a nonzero multiplication in $A_1\cdot A_1$ so $p\geq 1$. From \Cref{bounds_for_codim3} $p\leq 1$ and $q= 0$, thus $p=1$ and $R$ is of class $\H(1,0)$. 
 
 When $t=9$, \Cref{bounds_for_codim3} give $p=0$ and $q=0$ so $R$ is in class $\H(0,0)$. When $t\geq 10$, we have $\mu(t)\geq 3$ in \Cref{Betti_codim3}. Specifically either $\reg(R)>2$ or row one of the Betti table is trivial, i.e. $A_{ij}=0$ for all $j-i=1$. Applying \Cref{bounds_for_codim3} we get $p=q=0$ thus class $\H(0,0)$ for $t\geq 10$.
 \end{proof}

\subsection{Tor Algebras in codimension four}

In the remainder of the section we classify the Tor Algebras of the rings in \Cref{general AG}. 

Let $X$ be a finite set of points in projective space and write $F_X=\ell_1^{s}+\cdots \ell_t^{s}$ as in \eqref{eq: FX}. Under appropriate conditions it follows from results of Boij \cite{Boij} that the Betti numbers of $Q/\Ann(F_X)$ can be deduced from those of the coordinate ring of $X$ via a doubling construction.  Since the paper \cite{Boij} does not use the language of Macaulay inverse systems, we state a version of the results therein adapted to our terminology. The following results are applicable to arbitrary (not necessarily general) sets of points in projective space. See \cite[Theorem 4.8]{AG-WaringRk}  for a version of the following result specific to $R$ being the coordinate ring of general points.

\begin{prop}[Boij \cite{Boij}] \label{isDoubling}  
Let $X$ be a set of points in $\P^{N}$ with coordinate ring $R$. Adopt \eqref{eq: FX} and set $S=Q/\Ann(F_X)$. If $s\geq 2\reg(R)-1$ then $S$ is a doubling of $R$. Moreover, if $s\geq 2\reg(R)$, their Betti numbers are related by $$b_{ij}(S)=b_{ij}(R)+b_{N+1-i,s+N+1-j}(R).$$

\end{prop}
\begin{proof} 

 The statement regarding doubling is \cite[Theorem 3.4]{Boij}. The formula for the Betti numbers follows from \cite[Proposition 3.5]{Boij}. 
\end{proof}

\begin{example}\label{ex}
If $X$ is a set of points in $\P^3$ and the coordinate ring of $X$ has Betti table \eqref{Betti_codim3}, then from \Cref{isDoubling}, for $s\geq 2\reg(R)$ the ring $S=Q/\Ann(F_X)$  has Betti table as follows: 
\begin{equation}\label{Betti_codim4}
\begin{tabular}{r c c c c c}
 & 0 & 1 & 2 & 3 & 4 \\
total: & 1 & $ m+1$ & $2m$ & $m+1$ &1\\
0 : & 1 & . & . & .&. \\
$\mu-1$ : & . & $\alpha_1$ & $\alpha_2$ & $\alpha_3$ &.\\
$\mu$ : & . & $\beta_1$ & $\beta_2$ & $\beta_3$ &.\\
\\
$s-\mu$ : & . & $\beta_3$ & $\beta_2$ & $\beta_1$ &.\\
$s-\mu+1$ : & . & $\alpha_3$ & $\alpha_2$ & $\alpha_1$ &.\\
$s$:& .& .& . & . & 1
\end{tabular}
\end{equation}

We write $\mu$ in \eqref{Betti_codim4} instead of $\mu(t)$ when the number of points is unspecified. 
\end{example}

We use the Betti table in  \eqref{Betti_codim4}  to bound the structural constants of  $\Tor(S,k)$. Importantly we want to make sure that the multiplication within the two rows $\mu-1$ and $\mu$ cannot land in rows $s-\mu$ or $s-\mu+1$. 
This leads to the restriction on $s$ in the following lemma.

\begin{lemma} \label{bounds_for_codim4}
    Consider a ring with Betti table as in \eqref{Betti_codim4} resulting from doubling, with socle degree $s\geq 2\mu$ and $s\notin\{3\mu-3,3\mu-2,3\mu-1,3\mu\}$. 
    
    Then the invariants of its Tor algebra from \Cref{Notation Gor tor invariants} are bounded as follows
    \[
\begin{cases}
\begin{aligned}
\bar p
  &\leq
  \min\left\{\frac{\alpha_1(\alpha_1-1)}{2},\beta_2\right\}
  +\min\{\alpha_1\beta_3,\alpha_2\},\\
\bar q
  &\leq
  \min\{\alpha_1\alpha_2,\beta_3\}+\alpha_1
\end{aligned}
& \text{if }\mu=2\text{ and }\beta_3\neq0,\\[1ex]
\bar p=0,\quad \bar q=0
& \text{otherwise.}
\end{cases}
\]
Moreover, if either $\bar p=0$ or $\bar q=0$, then $\bar p=\bar q=0$.

\end{lemma}

\begin{proof}  First observe that as $s\notin\{3\mu-3,3\mu-2,3\mu-1,3\mu\}$ the multiplication $[A_i]_e\cdot [A_j]_d\sube [A_{i+j}]_{e+d}$ when $e-i,d-j\in\{\mu-1,\mu\}$ cannot have $(e+d)-(i+j)\in\{s-\mu,s-\mu+1\}$. Thus this multiplication has the same bounds as in \Cref{bounds_for_codim3}. 
   
    As $s\geq 2\mu$
    when $\mu>2$  \Cref{General-bounds} give $\bar p=0$ and $\bar q=0$.
    Note that if $\beta_3 =0$ then also $\beta_2=0$, as the rows labeled $s-\mu$, $s-\mu+1$ and $s$ of the Betti table \eqref{Betti_codim4} represent the Betti table of the canonical module $\omega_{R}$ in the notation of \Cref{isDoubling}. 
    Using this and the restrictions on $s$,
    
    \Cref{General-bounds} gives $\bar p =0$ and $\bar q =0$ when $\beta_3 =0$ and gives the bound on $\bar p$ when $\mu =2$ and $\beta_3\neq 0$. 
    This leaves the bound on $\bar q$ when $\mu=2$ and $\beta_3\neq 0$ as the only case to consider.  

    When $\mu=2$ and $\beta_3\neq 0$, applying \Cref{General-bounds} gives the bound 
    \[\bar q\leq \min(\alpha_1\alpha_2,\beta_3)+\min(\alpha_2\beta_3+\alpha_1\beta_2,\alpha_1).
    \]
   We will show that $\alpha_2\beta_3+\alpha_1\beta_2\geq \alpha_1$, thus giving the desired bound.
    If  $\beta_2>0$ the inequality $\alpha_1\beta_2+\alpha_2\beta_3 \geq \alpha_1$ is obvious. If 
    $\beta_2=0$ we can deduce using the Euler characteristic of the resolution in  \eqref{Betti_codim4} that $\alpha_2=\alpha_1+\beta_1+\beta_3+\alpha_3-1\geq \alpha_1$, as $\beta_3\neq 0$. Once again in this case we have $\alpha_1\beta_2+\alpha_2\beta_3 \geq \alpha_1\beta_3 \geq \alpha_1$.
    
    By Poincare duality \Cref{Mult_duality} shows that if $\bar p=0$ or $\bar q=0$ then both $A_1A_1=0$ and $A_1A_2=0$. 

\end{proof}

\begin{example}\label{ex: codim 4 p q bounds}
    Using the doublings of the Betti tables in \Cref{fig_Betti_tables} we can calculate the maximum values of $\bar p$ and $\bar q$:

\begin{center}
\begin{tabular}{c|c c c c c c}
        $t=$& $4$&5&6&7&8&9\\
    \hline
     $\bar p \leq$& 0&5&5&3&1&0  \\
     $\bar q\leq$& 0&6&6&3&2&0
\end{tabular}
\end{center}
\end{example}

\begin{lemma}\label{pValue}
Let $R=Q/J$ be a codimension-three Cohen--Macaulay ring, and
suppose that $S=Q/I$ is a doubling of $R$ arising from an exact
sequence
\[
0\longrightarrow \omega_R(-a)
 \stackrel{\psi}{\longrightarrow} R
 \longrightarrow S
 \longrightarrow 0.
\]
Assume that the mapping-cone resolution of $S$ associated to this
sequence is minimal. Set
\[
A=\Tor^Q(R,k)
\qquad\text{and}\qquad
A'=\Tor^Q(S,k).
\]
Then the natural homomorphism $R\longrightarrow S$ induces an
injective homomorphism of bigraded $k$-algebras
$\iota_*:A\hookrightarrow A'$ and consequently,
\[
\dim_k(A_iA_j)\leq \dim_k(A_i'A_j')
\qquad\text{for all }i,j\geq0.
\]
\end{lemma}

\begin{proof}
Let $G\to R$ be a minimal free resolution of $R$, and let
$L\to\omega_R(-a)$ be a minimal free resolution of the shifted
canonical module. Choose a comparison map
$
\widetilde{\psi}:L\longrightarrow G
$ 
lifting the injection $\psi:\omega_R(-a)\to R$. The mapping cone $F=\operatorname{Cone}(\widetilde{\psi})$
is a free resolution of $S$. As graded free modules,
$F_n=G_n\oplus L_{n-1}$.

The inclusions $\iota_n:G_n\hookrightarrow F_n$,
$\iota_n(g)=(g,0)$ form a chain map $\iota:G\to F$ lifting the natural surjection
$R\to S$. Since both $G$ and $F$ are minimal,
\[
A'_n = F_n\otimes_Qk
 \cong
(G_n\otimes_Qk)\oplus(L_{n-1}\otimes_Qk)=A_n\oplus(L_{n-1}\otimes_Qk),
\]
and the induced map
\[
\iota_*:
A=H(G\otimes_Qk)\longrightarrow H(F\otimes_Qk)=A'
\]
is the inclusion of the first summand, hence it is injective. 

The Tor product is natural with respect to homomorphisms of $Q$-algebras, so $\iota_*$ is an algebra homomorphism; see \cite[Chapter~XI,~2.1 and \S 4]{CartanEilenberg}. Therefore
\[
\iota_*(A_iA_j)
 =
\iota_*(A_i)\iota_*(A_j)
 \subseteq A_i'A_j'.
\]
Since $\iota_*$ is injective, the asserted dimension inequalities follow.
\end{proof}

\begin{remark}
In the applications below, the Betti number formula in \Cref{isDoubling} shows that the mapping cone has no cancellations and is therefore minimal. Thus the minimality hypothesis in \Cref{pValue} is satisfied.
\end{remark}

The following lemma can be inferred from \cite{Gor_Six_Gen}, however we state it here with hypotheses and terminology suitable for our purposes. 

\begin{lemma}\label{6-gens}
    Let $Q$ be a graded polynomial ring over an algebraically closed field $k$ of characteristic not equal to 2. 
    Let $I\subset Q$ be a codimension 4 Gorenstein ideal with $6$ minimal generators, then $Q/I$ is in class $\GH(5)$. 
\end{lemma}
\begin{proof}
  From  \cite{Gor_Six_Gen} $I=(L,y)$ where $L$ is the ideal generated by submaximal pfaffians of a $5\times 5$ skew-symmetric matrix and $y$ is a regular element of $Q/L$. Let $G$ be the minimal free resolution of $Q/L$ and $K$ the Koszul complex on $y$. The tensor product $F=G\otimes_QK$ has a DG algebra structure and is a resolution of $Q/I$. 
  
  By K\"unneth's Theorem $H(F\otimes_Q k)\cong H(G\otimes_Q k)\otimes_k H(K\otimes_Q k)$. As $Q/L$ is Gorenstein with 5 generators we know it is of class $\G(5)$ which gives the multiplication on $H(G\otimes_Q k)$. We also know that $Q/(y)$ is in class $\C(1)$ as it is of codepth one so the multiplication of $H(K\otimes_Q k)$ is that of an exterior algebra on one element. The multiplication on $H(F\otimes_Q k)$ can be recovered from these classes which shows $Q/I$ is in class $\GH(5)$.

\end{proof}

In the following theorem we employ the notation $d(t)$ from \eqref{eq: dt} as well as the fact that the invariants $\bar{p}$ and $\bar{q}$ in \Cref{Notation Gor tor invariants}
 suffice to classify Tor algebras of codimension 4 Gorenstein rings.

\begin{theorem}\label{AG codepth 4}

 Graded Artinian Gorenstein  rings $S$ of codimension 4 of Waring  rank $t$, with socle degree $s$ satisfying 
$2\mu(t)\leq s\leq3\mu(t)-4$ or $s\geq3\mu(t)+1$
that are general with respect to Betti numbers are classified by the value of $t$ as given in \Cref{full_theorem}. Their associated Tor algebras have maximum multiplication rank. 
\end{theorem}
\begin{proof}
 Write $S=Q/\Ann(F_X)$ with $Q$ a polynomial ring and $F_X=\ell_1^{s}+\cdots \ell_t^{s}$ as in \eqref{eq: FX}. 
 
By \Cref{general AG} the set $X$ is general with respect to Betti numbers, thus its coordinate ring $R $ has Betti table \eqref{Betti_codim3}. The regularity of $R$ is $\reg(R)=d(t)$ by \cite[Corollary 4.7]{EisenbudSyzygy} and the least degree of a relation of $R$ is $\mu(t)$ as in \eqref{eq: mu}. 

Since  $d(t)\leq \mu(t)$, the hypothesis $s\geq 2\mu(t)$ implies $s\geq 2d(t)$. Thus \Cref{isDoubling} gives that $S$ is doubling of $R$ and  by  \Cref{ex} $S$ has Betti table \eqref{Betti_codim4}. The inequalities relating $s$ and $\mu(t)$ ensure that the hypothesis of \Cref{bounds_for_codim4} is satisfied.

Since $S$ has codimension four, $R$ must have codimension $3$ which gives that $t\geq 4$.

When $t=4$ or $t\geq 9$ by \Cref{bounds_for_codim4} $\bar p =\bar q=0$ so the ring $S$ is of class $\GS$.

When $t=5$ or 6 from \Cref{isDoubling} and \Cref{fig_Betti_tables} we know $\Ann(F_X)$ will have $6$ minimal generators. By \Cref{6-gens} we obtain that $S$ is in class $\GH(5)$.

When $t=7$ from \Cref{ClassThmForJ} we know that $p=3$ for $R$. By \Cref{pValue} we know $\bar p\geq 3$, and by \Cref{bounds_for_codim4} we know $\bar p\leq 3$ so $\bar p=3$. When $\bar p=3$ there are two possible classes,  $\GT$ or $\GH(3)$. However for $\GH(3)$ we need $\bar q=4$ but from \Cref{bounds_for_codim4} we get $\bar q\leq 3$ so $S$ must be in class $\GT$.

Similarly when $t=8$, we know $1=p\leq \bar p$. By \Cref{bounds_for_codim4} we know $\bar p\leq 1$ so $\bar p=1$ and $S$ must be of class $\GH(1)$.  
\end{proof}

\section{Examples}

We showed that the Tor algebra classes of coordinate rings for general points in $\P^3$ are determined by the number of points. 
We have also shown maximum multiplication rank is achieved for coordinate rings of general points in $\P^3$  and their doublings of high enough socle degree, but this need not hold for other configurations of points in $\P^3$. 
\begin{definition}
    A set of $t$ points in $\P^N$ has {\it general Hilbert function} if its coordinate ring $R$ is $H_R(i)=\min({N+i\choose N},t)$.
\end{definition}

Using Macaulay2 and the package TorAlgebra \cite{TorAlgebraSource} we calculate some examples. 
\begin{example}\label{Ex: other points} Let $Q=k[x,y,z,w]$. The ideal 
\[
J=(zw,yw,xw,xz+z^2,xy+yz,y^2z-z^3)
\]
is the defining ideal of five points in $\P^3$ with three of them on a line.
The ideal \[U=(yw,xy+yz,xzw+zw^2,x^2z-2y^2z-xz^2-z^2w-zw^2,z^3w-zw^3,x^3w-xw^3)\]
the the defining ideal of eleven points in $\P^3$.

    The coordinate rings $Q/J$and $Q/U$ have Betti tables and Hilbert functions: 
    \begin{multicols}{2}
    
         $Q/J:\;$  \begin{tabular}{r c c c c}
      & 0 & 1 & 2 & 3 \\
    total: & 1 & 6 & 8 & 3\\
    0 : & 1 & . & . & . \\
    1 : & . & 5 & 6 & 2 \\
    2 : & . & 1 & 2 & 1 
    \end{tabular} 
    
    \vspace{10pt}
    
    $H_{Q/J}=(1,4,5,5,...)$
    
    \columnbreak
    
    $Q/U:\;\begin{matrix}
       & 0 & 1 & 2 & 3\\
      \text{total:} & 1 & 6 & 8 & 3\\
      0: & 1 & . & . & .\\
      1: & . & 2 & 1 & .\\
      2: & . & 2 & 2 & .\\
      3: & . & 2 & 5 & 3
      \end{matrix} $
      
       \vspace{5pt}
       
      $H_{Q/U}=(1,4,8,11,11...)$
    \end{multicols}

Applying \Cref{General-bounds} we get that $p
\leq 2$ and $q\leq 1$ for $Q/J$, and $p\leq 5$ and $q\leq 3$ for $Q/U$. However neither ring achieves these bounds as $Q/J$ has Tor algebra class $\H(0,0)$ which means $p=q=0$ and $Q/U$ is of class $\B$ with $p=q=1$.

A doubling of $J$ with socle degree 7 is the ideal 
\[I=(zw,yw,xw,xz+z^2,xy+yz,y^2z-z^3, 2y^6-3z^6,yz^6-2w^7,x^7-w^7)\sube Q.\] The Betti numbers of $Q/I$ are as in \Cref{isDoubling}.
Again applying \Cref{General-bounds} we get the bound $\bar p\leq 7$ and $\bar q\leq 6$ but these bounds are not achieved as it is of class $\GS$ which has $\bar p=\bar q=0$. 
\end{example}

This shows that ideals of points and their doublings need not have maximum multiplication, even for general Hilbert function as in $Q/J$.

The rings $Q/J$ and $Q/U$ in \Cref{Ex: other points} as well as the coordinate ring of four general points have total Betti numbers $(1,6,8,3)$. Both $Q/J$ and the coordinate rings of four general points are in class $\H(0,0)$, but $Q/U$ is of class $\B$, so total Betti numbers are not sufficient to classify ideals of points. However unlike the other two $Q/U$ does not not have general Hilbert function.

These examples lead to the question:

\begin{quest}
    Are the Tor algebra classes for coordinate rings of points in $\P^3$ with general Hilbert function determined by their total Betti numbers? 
\end{quest}

\bibliographystyle{alpha}
\bibliography{refs}

\end{document}